\documentclass[12pt]{article}
\usepackage[margin=1in]{geometry}
\usepackage[usenames]{color}
\usepackage[colorlinks=true,
linkcolor=webgreen,
filecolor=webbrown,
citecolor=webgreen]{hyperref}

\definecolor{webgreen}{rgb}{0,.5,0}
\definecolor{webbrown}{rgb}{.6,0,0}

\usepackage{color}

\usepackage{url}
\usepackage{graphicx}
\usepackage{subcaption}

\usepackage[T1]{fontenc}

\usepackage{amsmath}
\usepackage{amssymb}
\usepackage{amsthm}
\usepackage{mathrsfs}

\theoremstyle{plain}
\newtheorem{theorem}{Theorem}

\theoremstyle{definition}

\newtheorem{problem}{Problem}

\theoremstyle{remark}

\title{Prefixes of the Fibonacci word that end with a cube}
\author{Narad Rampersad\footnote{
Department of Math/Stats,
University of Winnipeg,
515 Portage Ave.,
Winnipeg, MB, R3B 2E9
Canada; {\tt narad.rampersad@gmail.com}.}}

\begin{document}
\maketitle
\begin{abstract}
  We study the prefixes of the Fibonacci word that end with a cube.
  Using Walnut we obtain an exact description of the positions of the
  Fibonacci word at which a cube ends.
\end{abstract}

\section{Introduction}
This paper is motivated by the following remarkable result, which was
originally conjectured by Jeffrey Shallit and proved by
Mignosi, Restivo, and Salemi \cite{MRS98}:
\begin{quote}
  \it An infinite word ${\bf w}$ is ultimately periodic if and only if
  all sufficiently long prefixes of ${\bf w}$ end with a repetition of
  exponent at least $\varphi^2$, where $\varphi$ is the golden ratio.
\end{quote}
The \emph{exponent} of a word is the ratio of its length to its
minimal period.  In particular, this result implies that no aperiodic
infinite word can have all sufficiently long prefixes end with a
\emph{cube} (a word with exponent $3$).  Counting the
number of prefixes of an infinite word that end with cubes can
therefore provide a measure, in some sense, of how close the infinite
word is to being ultimately periodic.

The first candidate that one would choose to investigate in regards to
this measure is the \emph{Fibonacci word}.  Indeed, Mignosi et al.\
also proved that the Fibonacci word witnesses the optimality of their
result in the following sense:
\begin{quote}
  \it For any $\epsilon >0$, all sufficiently long prefixes of
  the Fibonacci word
  $${\bf f} = 010010100100101001010010\cdots$$
  end with repetitions of exponent at least $\varphi^2-\epsilon$.
\end{quote}

In this paper we examine the positions at which a cube ends in the
Fibonacci word (the starting positions of cubes in the Fibonacci word
have been characterized by Mousavi, Schaeffer, and
Shallit~\cite{MSS16}).  Let ${\bf cubes_f}$ be the infinite word whose
$n$-th term is
\[
  \begin{cases}1 \text{ if a cube ends at position $n$ of ${\bf
        f}$,} \\
    0 \text{ otherwise.}\end{cases}
\]
For any $n\geq 0$, let $(n)_F$ denote the canonical representation of
$n$ in the Fibonacci (Zeckendorf) numeration system.

\begin{theorem}\label{thm:run1}
  There are arbitrarily long runs of $1$'s in ${\bf cubes_f}$.  More
  precisely, the runs of $1$'s in ${\bf cubes_f}$ are characterized by
  the following: If $(i)_F$ has the form
  \[
    (i)_F \in (10)^+0(0+10)(00)^*0w,
  \]
  where $w \in 0(10)^*(\epsilon+1)$ then
  ${\bf cubes_f}$ contains a run of $1$'s of length
  \begin{itemize}
  \item $F_{2n+2}-1$, if $|w|=2n$ for some $n \geq 0$,
  \item $F_{2n+3}-1$, if $|w|=2n+1$ for some $n \geq 0$,
  \end{itemize}
  beginning at position $i$.
\end{theorem}

\begin{theorem}\label{thm:run0}
  The runs of $0$'s in ${\bf cubes_f}$ have lengths $1$, $2$, $3$, $7$, $8$,
  and $13$.  The only run of length $13$ occurs at the beginning of
  ${\bf cubes_f}$.  For each of the other lengths ($1$, $2$, $3$, $7$, and
  $8$), there are infinitely many runs of that length in ${\bf cubes_f}$.
\end{theorem}

The proofs of these theorems are given in the next section.

\section{Walnut computations}
Our main results are all obtained by computer using Walnut
\cite{Walnut}.  We begin with the command
\begin{verbatim}
eval fib_end_cubes "?msd_fib Ei En n > 1 & j = i+3*n-1 &
    (Ak k < 2*n => F[i+k] = F[i+k+n])":
\end{verbatim}
which produces the automaton in Figure~\ref{fib_end_cubes}, which
accepts the Zeckendorf representations of the positions at which a cube
ends in ${\bf f}$.
  \begin{figure}[htb]
    \centering
    \includegraphics[scale=0.375]{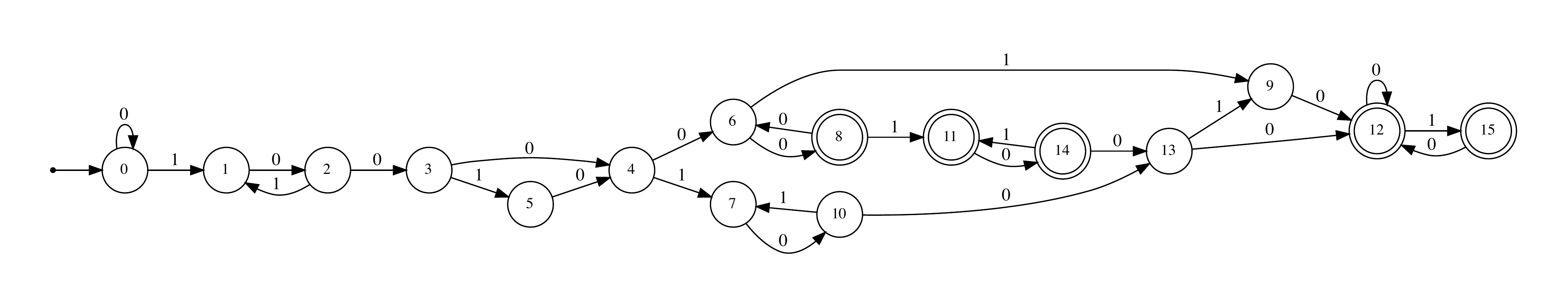}
    \caption{Automaton for ending positions of cubes in ${\bf f}$}\label{fib_end_cubes}
  \end{figure}

\begin{proof}[(Proof of Theorem~\ref{thm:run1}.)]
To determine the lengths of the runs of $1$'s in ${\bf cubes_f}$, we use the
command
\begin{verbatim}
eval fib_end_cubes_run "?msd_fib n>=1 & (At t<n =>
    $fib_end_cubes(i+t)) & ~$fib_end_cubes(i+n) &
    (i=0|~$fib_end_cubes(i-1))":
\end{verbatim}
which produces the automaton in Figure~\ref{fib_end_cubes_run}, which
accepts the Zeckendorf representations of pairs $(i,\ell)$ such that
there is a run of $1$'s in ${\bf cubes_f}$ of length $\ell$ starting at
position $i$.
  \begin{figure}[htb]
    \centering
    \includegraphics[scale=0.5]{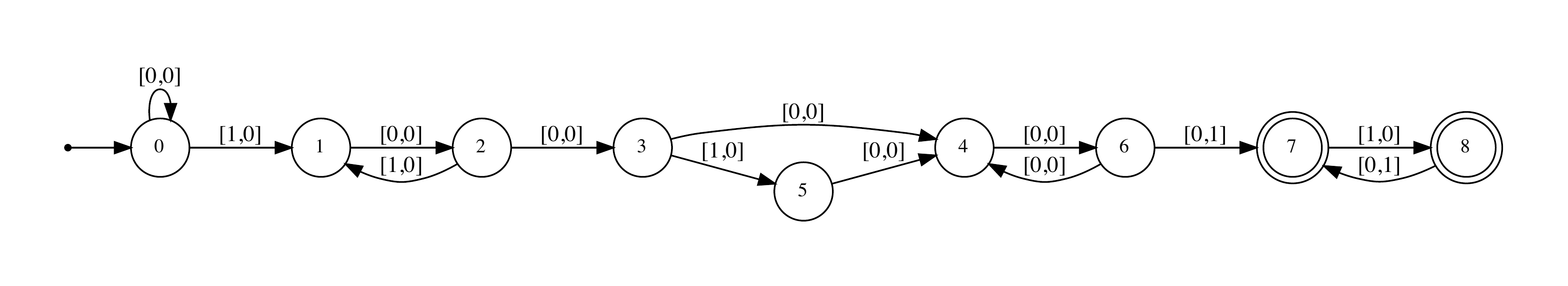}
    \caption{Automaton for runs of $1$'s in ${\bf cubes_f}$}\label{fib_end_cubes_run}
  \end{figure}

By examining the structure of this automaton we see that for an accepted
pair $(i,\ell)$, the representation $(i)_F$ has the form $(i)_F =
(10)^+0(0+10)(00)^*0w$, where $w \in 0(10)^*(\epsilon+1)$.
Furthermore, if $|w|=2n$, then $(\ell)_F = (10)^n$ and if $|w|=2n+1$,
then $(\ell)_F = (10)^n1$.  Now, let $F_m$ denote the $m$-th Fibonacci
number and recall the identities:
\[
\sum_{j=0}^{n-1}F_{2j+1} = F_{2n} \quad\text{ and }\quad
\sum_{j=1}^n F_{2j} = F_{2n+1} - 1.
\]
Hence, if $|w|=2n$, we have
$$\ell = \sum_{j=1}^n F_{2j+1} = F_{2n+1} + F_{2n} - F_1 =
F_{2n+2}-1$$
and if $|w|=2n+1$ we have
$$\ell = \sum_{j=1}^{n+1} F_{2j} = F_{2n+2} + F_{2n+1} - 1 =
F_{2n+3}-1.$$
\end{proof}

\begin{proof}[(Proof of Theorem~\ref{thm:run0}.)]
To determine the lengths of the runs of $0$'s in ${\bf cubes_f}$, we use the
command
\begin{verbatim}
eval fib_no_cubes_run "?msd_fib n>=1 & (At t<n =>
    ~$fib_end_cubes(i+t)) & $fib_end_cubes(i+n) &
    (i=0|$fib_end_cubes(i-1))":
\end{verbatim}
which produces the automaton in Figure~\ref{fib_no_cubes_run}, which
accepts the Zeckendorf representations of pairs $(i,\ell)$ such that
there is a run of $0$'s in ${\bf cubes_f}$ of length $\ell$ starting at
position $i$.
  \begin{figure}[htb]
    \centering
    \includegraphics[scale=0.35]{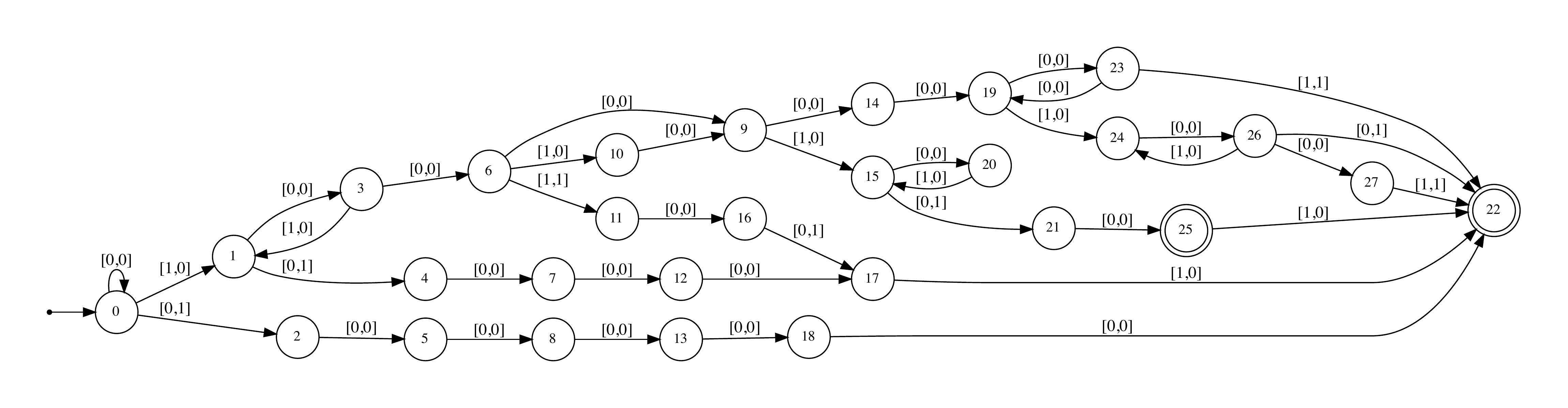}
    \caption{Automaton for runs of $0$'s in ${\bf cubes_f}$}\label{fib_no_cubes_run}
  \end{figure}

We can project this automaton onto the second component of its input
with the command
\begin{verbatim}
eval fib_no_cubes_run_length "?msd_fib Ei $fib_no_cubes_run(i,n)":
\end{verbatim}
which produces the automaton in Figure~\ref{fib_no_cubes_run_length}.
  \begin{figure}[htb]
    \centering
    \includegraphics[scale=0.5]{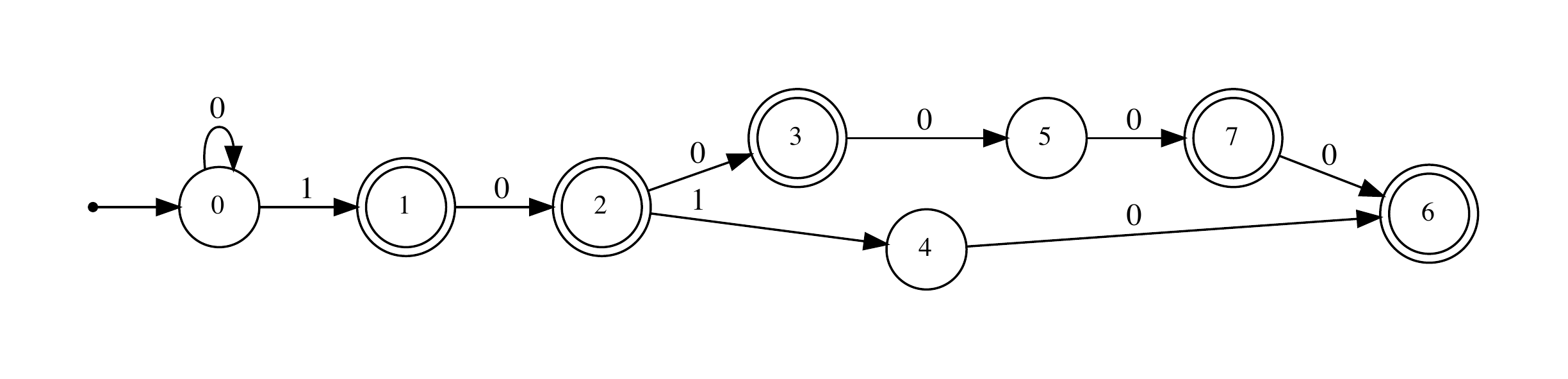}
    \caption{Automaton for lengths of runs of $0$'s in ${\bf cubes_f}$}\label{fib_no_cubes_run_length}
  \end{figure}
We see that the only possible run lengths are $\ell \in
\{1,2,3,7,8,13\}$.

The command
\begin{verbatim}
eval tmp "?msd_fib Ai Ej j>i & $fib_no_cubes_run(j,1)":
\end{verbatim}
evaluates to TRUE, indicating that there are infinitely many runs of
$0$'s of length $1$.  This is also the case for run lengths $2$, $3$, $7$,
and $8$.  For length $13$ however, we get a result of FALSE.
\end{proof}

The positions of the runs of length $7$ and $8$ have a simple
structure, so we describe these next.

\begin{theorem}\
  \begin{itemize}
  \item The runs of $0$'s in ${\bf cubes_f}$ of length $8$ begin at
    positions $i$ where $(i)_F \in (10)^+0001$.
  \item The runs of $0$'s in ${\bf cubes_f}$ of length $7$ begin at
    positions $i$ where $(i)_F \in (10)^+01001$.
  \end{itemize}
\end{theorem}

\begin{proof}
  These are obtained via the commands
\begin{verbatim}
eval tmp "?msd_fib $fib_no_cubes_run(j,8)":
eval tmp "?msd_fib $fib_no_cubes_run(j,7)":
\end{verbatim}
\end{proof}

The descriptions of the starting positions for the other lengths of
runs of $0$'s in ${\bf cubes_f}$ are a little more complicated, so we omit
them here, but the reader can easily compute these with Walnut.

\begin{theorem}
  The density of $0$'s in ${\bf cubes}_{\bf f}$ is zero.
\end{theorem}

\begin{proof}
We examine the complement of the automaton in
Figure~\ref{fib_end_cubes}.  The Walnut command
\begin{verbatim}
eval fib_no_end_cubes "?msd_fib ~$fib_end_cubes(j)":
\end{verbatim}
produces the automaton in Figure~\ref{fib_no_end_cubes},
which gives the positions in ${\bf f}$ where no cube ends.
  \begin{figure}[htb]
    \centering
    \includegraphics[scale=0.375]{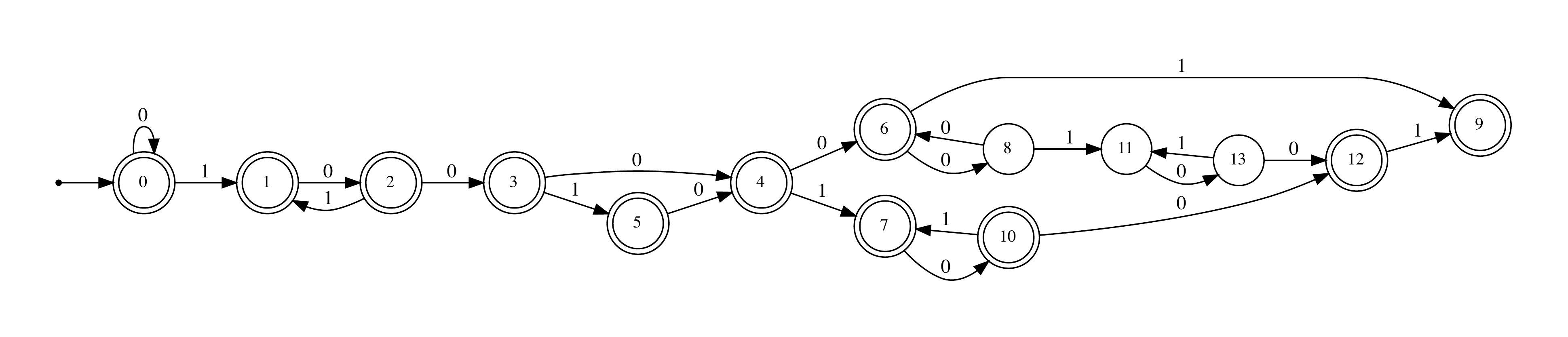}
    \caption{Automaton for positions in ${\bf f}$ where no cube ends}\label{fib_no_end_cubes}
  \end{figure}
  To complete the proof, it suffices to show that there are only
  polynomially many strings of length
  $n$ that are accepted by this automaton.  This can be seen directly
  from the structure of the automaton: since this automaton does not
  have two cycles that can both mutually reach each other, we can
  conclude that the number of strings of length
  $n$ accepted by this automaton is polynomially bounded (see, for
  example, \cite{GKRS08}).
\end{proof}

\section{Other Sturmian words}
Although the Fibonacci word is ``optimal'' with respect to the result
of Mignosi et al.\ mentioned in the Introduction, some computer
calculations suggest that there may be other Sturmian words that have
even more prefixes that end with cubes than the Fibonacci word.

For any infinite word ${\bf w}$, let us define ${\bf cubes}_{\bf w}$
to be the binary word whose $n$-th term is $1$ if
${\bf cubes}_{\bf w}$ has a cube ending at position $n$, and $0$
otherwise.  Let $\operatorname{max\_no\_cubes}({\bf w})$ denote the
largest $\ell$ such that ${\bf cubes}_{\bf w}$ contains infinitely
many runs of $0$'s of length $\ell$.  Let us also define
$S_{\bf w}(n)$ to be the sum of the first $n$ terms of
${\bf cubes}_{\bf w}$.  That is, $S_{\bf w}(n)$ counts the number of
positions $< n$ at which a cube ends in ${\bf w}$.

Now Theorem~\ref{thm:run0} shows that
$\operatorname{max\_no\_cubes}({\bf f})=8$.  Let ${\bf c}_\alpha$ be
the characteristic Sturmian word with slope $\alpha$.  It is not hard
to find a $\beta$ for which
$\operatorname{max\_no\_cubes}({\bf c}_\beta)=3$. Let
$\beta = (5-\sqrt{13})/6 = [0;4,\overline{3}]$.  Then ${\bf c}_\beta$
is a concatenation of the blocks $00001$ and $0001$, so for any given
position, there is always an occurrence of $000$ ending either at that
position or within the next $3$ positions.  Hence, we have
$\operatorname{max\_no\_cubes}({\bf c}_\beta)=3$.

Computationally, we can examine $S_{{\bf c}_\beta}(n)$ and
$S_{\bf f}(n)$ and compare these two quantities.
Table~\ref{tab:partial_sum} gives some values of these two functions.
Computer calculations show that $S_{{\bf c}_\beta}(n) > S_{\bf f}(n)$
for $2 \leq n \leq 3000$.
\begin{table}[htb!]
  \centering
\begin{tabular}{|c|c|c|}
  \hline
  $n$ & $S_{\bf   f}(n)$ & $S_{{\bf c}_\beta}(n)$ \\
  \hline
  500 & 353 & 408 \\
  1000 & 779 & 860 \\
  2000 & 1722 & 1812 \\
  3000 & 2669 & 2716 \\
  \hline
\end{tabular}
\caption{Comparing $S_{{\bf c}_\beta}(n)$ and  $S_{\bf   f}(n)$}
\label{tab:partial_sum}
\end{table}

We have the following open questions:

\begin{problem}
  Is it possible to determine $\operatorname {max\_no\_cubes}({\bf c}_\alpha)$
  from the continued fraction expansion of $\alpha$?
\end{problem}

\begin{problem}
  What is the least (resp.\ greatest) possible value of $\operatorname{max\_no\_cubes}({\bf
      c}_\alpha)$ over all $\alpha$?  Is it $3$ (resp.~$8$)?
\end{problem}

\begin{problem}
  Is there an $\alpha$ such that for all other $\alpha'$ the function
  $S_{{\bf c}_\alpha}(n)$ is eventually greater than
  $S_{{\bf c}_{\alpha'}}(n)$?
\end{problem}

\begin{problem}
  Can one prove that the density of $0$'s in ${\bf
    cubes}_{\mathbf{c}_\alpha}$ is $0$ for all $\alpha$?
\end{problem}

One might also wish to investigate the relationship between the
critical exponent of an infinite word ${\bf w}$ and the density of
$0$'s in ${\bf cubes}_{\bf w}$.  The critical exponent of ${\bf w}$ is
the quantity
\[
  \sup \{r : {\bf w} \text{ contains a factor with exponent } r\}.
\]
Note that it is easy to construct an aperiodic word with unbounded
critical exponent for which ``almost all'' positions are the ending
position of a cube: for example, the infinite word
\[
  010^210^410^810^{16}10^{32}1\cdots
\]
has this property.  So it is natural to restrict our attention to
words with bounded critical exponent.  The Fibonacci word has critical
exponent $2+\varphi \approx 3.618$, and all Sturmian words have
critical exponent at least this large.  Are there words ${\bf w}$ with
lower critical exponent for which the density of $0$'s in
${\bf cubes}_{\bf w}$ is still $0$?  The answer is ``yes''.  For
instance, the fixed point ${\bf x}$ (starting with $0$) of the
morphism $0 \to 0001$, $1 \to 1011$ has critical exponent
$10/3$~\cite[p.~99]{Krie08}, and just as we did for the Fibonacci
word, we can use Walnut to show that the density of $0$'s in
${\bf cubes}_{\bf x}$ is $0$ (after computing the automaton for the
$0$'s in ${\bf cubes}_{\bf x}$, one computes the eigenvalues of the
adjacency matrix and finds that they are all strictly smaller than
$4$).

\begin{problem}
  What is the infimum of the critical exponents among all infinite
  words ${\bf w}$ for which the density of $0$'s in
  ${\bf cubes}_{\bf w}$ is $0$?  Is it $3$?
\end{problem}

\section*{Acknowledgments}
The idea for this work came from discussions with James Currie.  We
thank him for those stimulating conversations.

\end{document}